\documentclass[11pt]{amsart}
\usepackage{amsmath,amssymb,amsthm}
\usepackage[latin1]{inputenc}
\usepackage{version,tabularx,multicol}
\usepackage{graphicx,float,psfrag}
 \usepackage{stmaryrd}
\usepackage{mathrsfs,eufrak}

\newcommand{\reff}[1]{(\ref{#1})}

\theoremstyle{plain}
\newtheorem*{theo*}{Theorem}
\newtheorem*{cor*}{Corollary}
\newtheorem*{conj*}{Conjecture}
\newtheorem{theo}{Theorem}[section]
\newtheorem{cor}[theo]{Corollary}
\newtheorem{prop}[theo]{Proposition}
\newtheorem{lem}[theo]{Lemma}

\theoremstyle{remark}

\newcommand{\cc}{{\mathcal C}}

\newcommand{\cf}{{\mathcal F}}

\newcommand{\ck}{{\mathcal K}}

\newcommand{\cn}{{\mathcal N}}
\newcommand{\cm}{{\mathcal M}}

\newcommand{\ct}{{\mathcal T}}

\newcommand{\cz}{{\mathcal Z}}

\newcommand{\bF}{\mathbb{\mathcal F}}

\newcommand{\E}{{\mathbb E}}

\newcommand{\N}{{\mathbb N}}
\renewcommand{\P}{{\mathbb P}}

\newcommand{\R}{{\mathbb R}}

\newcommand{\T}{{\mathbb T}}

\newcommand{\bm}{{\mathbf m}}
\newcommand{\bN}{{\mathbf N}}

\newcommand{\TT}{{\mathfrak T}}

\newcommand{\ind}{{\bf 1}}

\newcommand{\inv}[1]{\mathop{\frac{1}{ #1}}\nolimits}
\newcommand{\expp}[1]{\mathop {\mathrm{e}^{ #1}}}

\newcommand{\lb}{[\![}
\newcommand{\rb}{]\!]}

\begin{document}

\title{The forest associated with the record process on a L\'evy tree}

\date{\today}
\author{Romain Abraham} 

\address{
Romain Abraham,
Laboratoire MAPMO, CNRS, UMR 7349,
F\'ed\'eration Denis Poisson, FR 2964,
 Université d'Orléans,
B.P. 6759,
45067 Orléans cedex 2,
France.
}
  
\email{romain.abraham@univ-orleans.fr}

\author{Jean-François Delmas}

\address{
Jean-Fran\c cois Delmas,
Université Paris-Est, \'Ecole des Ponts, CERMICS, 6-8
av. Blaise Pascal, 
  Champs-sur-Marne, 77455 Marne La Vallée, France.}

\email{delmas@cermics.enpc.fr}

\thanks{This work is partially supported by the ``Agence Nationale de
  la Recherche'', ANR-08-BLAN-0190.}

\begin{abstract}
  We  perform  a pruning  procedure  on a  L\'evy  tree  and instead  of
  throwing away  the removed sub-tree, we  regraft it on  a given branch
  (not related to the Lévy tree).  We prove that the tree constructed by
  regrafting is distributed as  the original L\'evy tree, generalizing a
  result of Addario-Berry, Broutin and Holmgren where  only Aldous's tree  is considered. As a  consequence, we
  obtain that the ``average pruning time'' of a leaf is distributed  as the
  height of a leaf picked at random in the L\'evy tree.
\end{abstract}



\maketitle

\section{Introduction}
\label{sec:intro}

L\'evy trees arise as the scaling limits of Galton-Watson trees in the
same way as continuous state branching processes (CSBPs) are the scaling
limits of Galton-Watson processes (see \cite{dlg:rtlpsbp}, Chapter
2). Hence, L\'evy trees can be seen as the genealogical trees of some
CSBPs, \cite{lglj:bplpep}. One can define a random variable
$\ct$ in the space of real trees (see
\cite{epw:rprtrgr,e:prt,dlg:pfalt}) that describes the
genealogy of a CSBP with branching mechanism $\psi$ of the form:
\[
\psi(\lambda)=\alpha\lambda+\beta\lambda^2
+\int_{(0,+\infty)}\left(\expp{-\lambda 
  r}-1+\lambda r\right)\pi(dr)  \quad\text{for } \lambda\geq 0, 
\]
with $\alpha\geq 0$, $\beta\ge 0$ and $\pi$ a $\sigma$-finite measure on
$(0,+\infty)$        such        that        $\int_{(0,+\infty)}(r\wedge
r^2)\pi(dr)<+\infty$.     We   assume    that   either    $\beta>0$   or
$\pi((0,1))=+\infty  $.   In   particular,  the  corresponding  CSBP  is
sub-critical  as  $\psi'(0)=\alpha\ge 0$ . In
order  to  use   the  setting  of  measured  real   trees  developed  in
\cite{adh:etiltvp},  we shall  restrict  ourselves to  compact Lévy
trees, that is with branching mechanism satisfying the Grey condition:
\[
\int^{+\infty } \frac{dv}{\psi(v)} <+\infty .
\]
This condition is equivalent to the compactness of the Lévy tree, and
to the a.s. extinction in finite time of the corresponding CSBP.

In \cite{adv:plcrt}, a pruning mechanism has been constructed
so that the Lévy tree with branching mechanism $\psi$ pruned at rate $q>0$  is a L\'evy
tree  with branching mechanism $\psi_q$ defined by:
\[
\psi_q(\lambda)=\psi(\lambda+q)-\psi(q) \quad\text{for } \lambda\geq 0.
\]
This pruning is performed by throwing marks on the tree in a
Poissonian manner and by cutting the tree according to these
marks, generalizing the fragmentation procedure of the Brownian tree introduced in \cite{ap:sac}. This pruning procedure allowed to construct a tree-valued Markov
process \cite{ad:ctvmp} (see also \cite{ap:tvmcdgwp} for an analogous
construction in a discrete setting) and to study the record process on Aldous's
continuum random tree (CRT) \cite{ad:rpcrt} which is related to the number
of cuts needed to reduce a Galton-Watson tree.

This problem of cutting down a random tree arises first in
\cite{mm:cdrt}: consider a rooted discrete tree with $n$ vertices, pick an
edge uniformly at random and remove it together with the sub-tree
attached to it and then iterate the procedure on the remaining tree
until only the root is left. The question is ``How many cuts are needed
to isolate the root by this procedure''? Asymptotics in law for this
quantity are given in \cite{mm:cdrt} when the tree is a Cayley tree
(see also \cite{b:ft,bm:ctlgwtbcrt} in this case where the problem is
generalized to the isolation of several leaves and not only the root)
and in \cite{j:rcrdrt} for conditioned (critical with
finite variance) Galton-Watson trees. A.s. convergence has also been
obtained in the latter case for a slightly different quantity in
\cite{ad:rpcrt} using a special pruning procedure that we describe
now.\\

Let  $\ct$  be  a  L\'evy  tree  with  branching  mechanism  $\psi$  and
$\bm^\ct(dx)$ its ``mass measure'' supported  by the leaves of $\ct$.  We
denote by $\P^\psi_r$ the distribution of the Lévy tree corresponding to
the CSBP with branching mechanism $\psi$ starting at $r$ and by
$\N^\psi$ the corresponding excursion measure also called canonical
measure (in particular, $\P^\psi_r$ can be seen as the distribution of
a ``forest'' of L\'evy trees given by a Poisson point measure with
intensity $r\N^\psi$). The
branching points  of the Lévy  tree are  either binary or  of infinite degree
(see \cite{dlg:pfalt}, Theorem 4.6) and to  each infinite degree branching point $x$,
one can  associate a  size $\Delta_x$ which  measures in some  sense the
number  of sub-trees  attached to  it (see  \reff{DefMas}  in Section
\ref{sec:loc.time}).   We then  consider  a measure  $\mu^\ct$ on  $\ct$
defined by:
\[
\mu^\ct(dy)=2\beta \ell^\ct(dy)+\sum_{x\in
  \mathrm{Br}_\infty(\ct)}\Delta_x\delta_x(dy),
\]
where  $\ell^\ct$ is the  length measure  on the  skeleton of  the tree,
$\mathrm{Br}_\infty(\ct)$  is the  set of  branching points  of infinite
degree and $\delta_x$  is the Dirac measure at  point $x$.  Aldous's CRT
corresponds  to   the  distribution  of  $\ct$   under  $\N^\psi$,  with
$\psi(\lambda)=\inv{2}\lambda^2$, and conditionally on $\bm^\ct(\ct)=1$.
In   this    case   $\mathrm{Br}_\infty(\ct)$   is    empty   and   thus
$\mu^\ct(dy)=\ell^\ct(dy)$.

Then  we consider, conditionally  given $\ct$,  a Poisson  point process
$M^\ct(d\theta,dy)$ 
of marks on the tree with intensity
$$\ind_{[0,+\infty)}(\theta)d\theta\, \mu^\ct(dy).$$
Parameter $y$ indicates the location of the mark whereas $\theta$
represents the time at which it appears.
For every $x\in\ct$, we set
\begin{center} {\sl $\theta(x)$ the first time $\theta$ at
which a mark appears between $x$ and the root.}
\end{center}

We consider $\Theta$ the average
of these first cutting times over the Lévy tree:
\[
\Theta=\int_{\ct}\theta(x)\bm^\ct(dx).
\]

It has  been proven  in \cite{ad:rpcrt} (Theorem  6.1 and  Corollary 5.3
with $\psi(u)=u^2/2$) in  the framework of the Aldous's  CRT, that if we
denote by  $X_n$ the number  of cuts needed  to isolate the root  in the
sub-tree   spanned   by   $n$   leaves  randomly   chosen,   then   a.s.
$\lim_{n\rightarrow+\infty  } X_n/L_n=\Theta$, with  $L_n\sim \sqrt{2n}$
the total length of the sub-tree.  Moreover, the law of $\Theta$ in that
case    is    a    Rayleigh    distribution    (i.e.     with    density
$x\expp{-x^2/2}\ind_{\{x\ge 0\}}$). The distribution of $\Theta$ is also
the law of the height of a leaf picked at random in Aldous's tree.  This
surprising  relationship  is  explained  by Addario-Berry,  Broutin  and
Holmgren  in  \cite{abbh:cdtmc}, Theorem  10.   The  authors consider  a
branch with  length $\Theta$, and when  a mark appears, the  tree is cut
and the sub-tree which does not  contain the root is removed and grafted
on  this branch  (the grafting  position is  described using  some local
time).  Then the  new tree obtained by this  grafting procedure is again
distributed as Aldous's tree.

The goal of this paper is  to generalize this result to  general L\'evy
trees. We  consider a L\'evy  tree $\ct$ under $\N^\psi$ 
and we perform the pruning procedure  described above. When a mark appears,
we  remove the  sub-tree attached  to this  mark and  keep  the sub-tree
containing the root. We denote by $\ct_q$ the resulting tree at time $q$
i.e. the  set of points  of the initial  tree $\ct$ which have  no marks
between them and the root  at time $q$: 
\[
\ct_q=\{x\in \ct; \theta(x)\geq q\}.
\]
According  to \cite{adv:plcrt} Theorem 1.1, $\ct_q$  is a  Lévy tree  with branching
mechanism $\psi_q$.   We consider $\Theta_q$ the average  of the records
shifted by $q$ over the Lévy tree $\ct_q$:
\[
\Theta_q=\int _{\ct_q}  (\theta(x)-q)\, \bm^\ct(dx).
\]
Remark that a.s. $\ct_q\subset \ct$ and hence $\Theta_q\le \Theta$; 

\begin{center}
\begin{figure}[h]
\psfrag{1}{$1$}
\psfrag{2}{$2$}
\psfrag{3}{$3$}
\psfrag{0}{$\emptyset$}
\psfrag{Theta1}{$\Theta_{\theta_1}$}
\psfrag{Theta2}{$\Theta_{\theta_2}$}
\psfrag{Theta3}{$\Theta_{\theta_3}$}
\scalebox{.8}{\includegraphics[width=10cm]{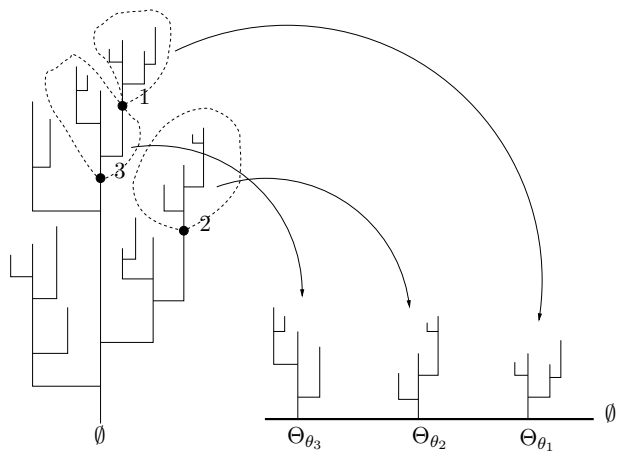}}
\caption{Pruning of  a L\'evy  tree (left) and  tree $\ct^R$  obtained by
  regrafting on a  branch (right).  The marks are  numbered according to
  their order of appearance.}
\label{fig:cutting}
\end{figure}
\end{center}

We define  an equivalence relation  on the tree  $\ct$: $x\sim y$  if the
function $\theta$ remains  constant on the path between  $x$ and $y$. We
consider  the  equivalence classes  $(\ct^i,  i\in  I^R)$  and denote
for each $i\in I^R$ by
$\theta_i$ the  common value  of the function  $\theta$. In  the pruning
procedure described above, the  tree $\ct^i$ corresponds to the sub-tree
which is removed  at time $\theta_i$ and it  is distributed according to
$\N^{\psi_{\theta_i}}$.   Then  we consider  a  branch  $B^R$ of  length
$\Theta$ rooted at some end point, say $\emptyset$. The sub-tree $\ct^i$
is grafted on  $B^R$ at distance $\Theta_{\theta_i}$ from  the root, see
Figure  \ref{fig:cutting}.  Let  $\ct^R$  denote this  tree obtained  by
regrafting.  Our main  result,  see Theorem  \ref{theo:main}, relies  on
Laplace transform computations and can be stated as follows.
\begin{theo*}
  Assume  the  Grey  condition  holds.  Under  $\N^\psi$,  $(B^R,  \ct^R)$
  is distributed  as  $(B,  \ct)$  where  $B$  is a  branch  from  the  root
  $\emptyset$ to a leaf chosen at  random on $\ct$ according to the mass
  measure $\bm^\ct$.
\end{theo*}

In particular, this theorem implies the following corollary.

\begin{cor*}
Under $\N^\psi[d\ct]$,  $\Theta$  is distributed  as the
height $H$ of a leaf of the L\'evy tree chosen at random according to
the mass measure $m^\ct$.
\end{cor*}

A probabilistic interpretation of those  results for the Brownian CRT is
provided in  \cite{abbh:cdtmc} using  a path transformation  on Brownian
bridge or in  \cite{b:ft} using a fragmentation tree. We  do not know if
such an approach is valid in the present general framework.

Using the
Bismut decomposition of Lévy trees, we recover and extend to general
Lévy trees Proposition 8.2 from
\cite{ad:ctvmp} on the asymptotics of the masses of $(\ct^i, i\in
I^R)$. For $i\in I^R$, set $\sigma^i=
\bm^\ct(\ct^i)$. 
\begin{cor*}
 Assume  the  Grey  condition  holds.    $\N^\psi$-a.e., we have:
\[
\lim_{\varepsilon\to 0}\frac{1}{\N^\psi[\sigma>\varepsilon]}
\sum_{i\in I^R}\ind_{\{\sigma^i\geq 
  \varepsilon\}}=\Theta.
\]
\end{cor*}
Similar      results      hold      for     the      convergence      of
$\inv{\N^\psi[\sigma\ind_{\{\sigma\leq    \varepsilon\}}]}    \sum_{i\in
  I^R}\sigma^i\ind_{\{\sigma^i\le   \varepsilon\}}$  to   $\Theta$,  see
Corollary \ref{cor:asymp-ti}. Those  results generalize Proposition 8.3 from
\cite{ad:ctvmp}.




\begin{center}
\begin{figure}[h]
\psfrag{0}{$\emptyset$}
\scalebox{.5}{\includegraphics[width=10cm]{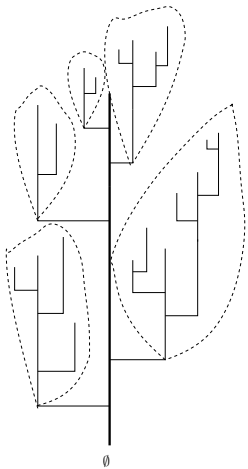}}
\caption{Bismut decomposition of a Lévy tree.}
\label{fig:bismut}
\end{figure}
\end{center}

The above theorem states  that the point process with atoms
$(\Theta_{\theta_i},\ct^i),\ i\in I^R$ is distributed as the point process that
appears in the Bismut decomposition of a L\'evy tree. This may seem
quite surprising. Indeed, if $\theta_i\le \theta_j$, then $\ct^i$ is
stochastically greater than $\ct^j$ (as a tree distributed according to
$\N^{\psi_q}$ can be obtained from a tree distributed according
$\N^{\psi_{q'}}$ for $q\ge q'$ by pruning). Consequently, the trees
that are grafted on $B^R$ are in some sense smaller and smaller
whereas the trees in the Bismut decomposition have the same
law. However, the intensity of the grafting is not uniform in the
first case (contrary to the Bismut decomposition) and depends on the
size of the trees grafted before, which gives at the end the identity
in distribution.

In the present work, we ignore  the marks that fall on the sub-trees once
they  have been  removed.  However,  we could  use them  to  iterate our
construction on each sub-trees $(\ct^i, i\in I^R)$ and so on, in order to
generalize to general L\'evy trees the result obtained for Aldous's CRT by
Bertoin and Miermont \cite{bm:ctlgwtbcrt}.\\

In  view of  the present  work, we  conjecture that  similar  results to
\cite{j:rcrdrt} hold  for infinite variance  offspring distribution. Let
$X_n$ denote the number of cuts needed to isolate the root by pruning at
edges a Galton-Watson  tree conditioned to have $n$  vertices.  We also
consider the pruning at  vertices inspired by \cite{adh:pgwttvmp}, which
is  the  discrete analogue  of  the  continuous  pruning: pick  an  edge
uniformly  at random and  remove {\bf  the vertex  from which  the edge
  comes from} together  with the sub-tree attached to  this vertex.  Let
$\tilde X_n$  be the  number of cuts  until the  root is removed  by this
procedure for  a Galton-Watson tree  conditioned to have  $n$ vertices.
According to  \cite{j:rcrdrt}, the number  of cuts needed to  remove the
root for the  pruning at vertices (that is $\tilde X_n$) or  to isolate the root
for  the pruning  at  edges (that is $X_n$)  are  asymptotically equivalent  for
finite variance offspring distribution.   However, we expect a different
behavior   in  the   infinite  variance   case.   Consider   a  critical
Galton-Watson  tree  with  offspring   distribution  in  the  domain  of
attraction of  a stable  law of index  $\gamma\in (1,2]$.   According to
\cite{d:ltcpcgwt}  or  \cite{k:spdtcpcgwt}, the  (contour  process  of  the)
Galton-Watson  tree  conditioned to  have  total  progeny $n$,  properly
rescaled, converges in  distribution to (the contour process of)
a   Lévy  tree   under  $   \N^{\psi}\left[\,\cdot\,|\sigma=1\right]$,  with
$\psi(\lambda) =  c_0\lambda^\gamma$ for some $c_0>0$.

\begin{conj*}
Let $L_n$  denote the length  of the
rescaled Galton-Watson tree conditioned  to have total progeny $n$. We
conjecture that: 
\[
\frac{\tilde X_n}{L_n}
\xrightarrow[n\rightarrow+\infty ] {(d) }Z,
\]
for some random variable $Z$ distributed  as the height of a leaf
chosen    at   random    according   to    the   mass    measure   under
$\N^{\psi}[\,\cdot\,|\sigma=1]$.
\end{conj*}

Set   $a=(\gamma-1)/\gamma$.  Using   Laplace  transform   (see  Theorem
\ref{theo:Bismut}), we get that the height $H$ of a leaf randomly chosen
in the L\'evy tree is  distributed under $\N^\psi$ as $\sigma^a Z$, with
$Z$   and  $\sigma$  independent   and  the   distribution  of   $Z$  is
characterized for $n\in \N$ by:
\[
\E\left[Z^n\right]= \inv{c_0^{n/\gamma} \gamma^n}
\frac{\Gamma(a)\Gamma(n+1)}{\Gamma(a(n+1))}\cdot
\]
In the particular  case of the Aldous CRT,  $\gamma=2$ and $c_0=1/2$, we
recover, using the  duplication formula of the gamma  function, that $Z$
(and thus $H$ under $\N^{\psi}[\,\cdot\,|\sigma=1]$) has Rayleigh
distribution.

The paper is organized as follows. We collect results on Lévy trees
in Section \ref{sec:levy}, with the Bismut decomposition is Section
\ref{subsec:Bismut} and the pruning procedure in Section \ref{sec:pruning}. The
main result is then precisely stated in Section \ref{sec:result} and
proved in Section \ref{sec:proof}.

\section{L\'evy trees and the forest obtained by pruning}
\label{sec:levy}
\subsection{Notations}

Let $(E,d)$ be a metric Polish space. For $x\in E$, $\delta_x$ denotes
the Dirac measure at point $x$. For $\mu$ a Borel measure on $E$ and $f$
a non-negative measurable function, we set $\langle \mu,f \rangle =\int
f(x) \, \mu(dx)= \mu(f)$. 

\subsection{Real trees}

We refer to \cite{c:ilt, dmt:tt, t:rtsds} for a general presentation
of $\R$-trees and to \cite{e:prt} or  \cite{lg:rta} for their
applications in the field of
random  real trees.  Informally,  real trees  are metric  spaces without
loops, locally  isometric to  the real line.   More precisely,  a metric
space $(T,d)$ is a real tree if the following properties are satisfied:
\begin{enumerate}
	\item For every $s,t\in T$, there is a unique isometric map $f_{s,t}$
from $[0,d(s,t)]$ to $T$ such that $f_{s,t}(0)=s$ and $f_{s,t}(d(s,t))=t$. 
	\item For every $s,t\in T$, if $q$ is a continuous injective map from
$[0,1]$ to $T$ such that $q(0)=s$ and $q(1)=t$, then
$q([0,1])=f_{s,t}([0,d(s,t)])$. 
\end{enumerate}
If $s,t\in T$, we will denote by $\llbracket s,t \rrbracket$ the range of the
isometric map  $f_{s,t}$ described above. We will  also write $\llbracket
s,t \llbracket$ for the  set $\llbracket s,t \rrbracket \setminus\{t\}$.

We say that $(T,d,\emptyset)$ is a rooted real tree with root
$\emptyset$ if $(T,d)$ is a real tree and $\emptyset\in T$ is a
distinguished vertex. 

Let $(T,d,\emptyset)$ be a rooted real tree.  If $x\in T$, the degree of
$x$, $n(x)$, is the number of connected components of $T\setminus\{x\}$.
We shall consider the set of leaves $ \mathrm{Lf}(T)=\{x\in T\backslash
\{\emptyset\},\,    n(x)=1\}$,    the    set   of    branching    points
$\mathrm{Br}(T)=\{x\in  T,  \, n(x)\ge  3\}$  and  the  set of  infinite
branching points is $\mathrm{Br}_\infty(T) =  \{ x\in T,\, n(x) = \infty
\} $.  The skeleton of $T$ is  the set of points in the tree that aren't
leaves: $\mathrm{Sk}(T)=T\backslash  \mathrm{Lf}(T)$.  The trace  of the
Borel $\sigma$-field of $T$  restricted to $\mathrm{Sk}(T)$ is generated
by    the    sets    $\llbracket    s,s'    \rrbracket$;    $s,s'    \in
\mathrm{Sk}(T)$.  Hence, one  defines uniquely  a  $\sigma$-finite Borel
measure $\ell^{T}$ on $T$, called length measure of $T$, such that:
\[ 
\ell^{T}(\mathrm{Lf}(T)) = 0
\quad\text{and}\quad
 \ell^{T}(\llbracket s,s'
\rrbracket)=d(s,s'). 
\]

For every $x\in T$, $\lb\emptyset ,x\rb$ is
interpreted as the ancestral line of vertex $x$ in the tree. We define
a partial order on $T$ by setting $x\preccurlyeq y$ ($x$ is an
ancestor of $y$) if $x\in\lb \emptyset,y\rb$.
If $x,y\in T$, there exists a unique $z\in T$, called the Most Recent
Common Ancestor (MRCA) of $x$ and $y$,  such that $\lb
\emptyset,x\rb\cap\lb\emptyset,y\rb=\lb\emptyset,z\rb$. We write
$z=x\wedge y$. 

\subsection{Measured rooted real trees}
We call an w-tree a weighted rooted real tree, i.e. a quadruplet $(T, d,
\emptyset, \bm)$  where $(T,d, \emptyset)$  is a locally  compact rooted
real tree and $ \bm$ is a locally finite measure on $T$.  Sometimes, we
will write $(T,  d^T, \emptyset^T, \bm^T)$ for $(T,  d, \emptyset, \bm)$
to stress the dependence in $T$,  or simply $T$ when there is no
confusion. We denote by $\TT$ the set of w-trees.

In order to define a tractable distance on w-trees, we need an
equivalence relation between two w-trees, i.e. we identify two w-trees
$(T,  d^T, \emptyset^T, \bm^T)$ and $(T',  d^{T'}, \emptyset^{T'},
\bm^{T'})$ if there exists an isometric function which maps $T$ onto
$T'$, which sends $\emptyset^T$ onto $\emptyset^{T'}$ and which
transports measure $m^T$ on measure $m^{T'}$.
We   will  denote   by   $\T$  the   set   of  measure-preserving   and
root-preserving isometry classes of w-trees.
One can define a topology on  $\T$ such that $\T$
is a Polish space, see for example \cite{ew:sprrrtmp,
  gpw:cdrmmslcmt,adh:ghptwms}.

Let $T,T'\in\TT$ be  w-trees that belong to the  same equivalence class.
Let $\varphi$ be a measure-preserving root-preserving isometry that maps
$T$ onto $T'$.  A w-tree-valued function $F$ of  the form $F(T,(x_i,i\in
I))$ where  $(x_i,i\in I)$ is a  family of points  of $T$ is said  to be
$\T$-compatible  if $F(T',(\varphi(x_i),i\in  I))$ belongs  to  the same
equivalence class as $F(T,(x_i,i\in I))$.

Let    $T\in   \TT$.  For $x\in T$, we set $h(x)=d(\emptyset, x)$ the
height of $x$ and   
\begin{equation}\label{eq:def-hmax}
H_{\text{max}}(T)=\sup_{x\in
  T} h(x)
\end{equation}
the height  of  the tree  (possibly infinite). Remark that for two
w-trees in the same equivalence class, the heights are the same,
hence $H_{\text{max}}(T)$ is well-defined for $T\in\T$.

For
$a\geq 0$, we set:
\[
T(a)=\{x\in T, \, d(\emptyset,x)=a\}
\quad\text{and}\quad
\pi_a(T)=\{x\in T, \, d(\emptyset,x)\leq a\},
\]
the restriction of the tree $T$  at level $a$ and the truncated tree $T$
up to level  $a$. We consider $\pi_a(T)$ with  the induced distance, the
root  $\emptyset$ and  the mass  measure $\bm^{\pi_a(T)}$  which  is the
restriction  of $\bm^T$  to $\pi_a(T)$,  to get  an w-tree. Let us
remark that the map $\pi_a$ is $\T$-compatible.  We denote by  $(T^{i,\circ},i\in I)$ the connected components of
$T\setminus  \pi_a(T)$.   Let  $\emptyset_i$  be  the MRCA  of  all  the
points   of    $T^{i,\circ}$.    We    consider    the   real    tree
$T^i=T^{i,\circ}\cup\{\emptyset _i\}$ rooted at point $\emptyset_i$ with
mass  measure  $\bm^{T^i}$ defined  as  the  restriction  of $\bm^T$  to
$T^i$. We will consider the point measure on $T\times \T$:
\[
\cn_a^T=\sum_{i\in I}\delta_{(\emptyset_i,T^i)}.
\]

\subsection{Grafting procedure}

We  will define  in this  section a  procedure by  which we  add (graft)
w-trees   on  an  existing w-tree. More precisely, let $T \in  \TT$ and let $((T_i,x_i),i\in I)$ be a
finite or  countable family of elements  of $\T\times T$.  We define the
real tree obtained by grafting the trees $T_i$ on $T$ at point $x_i$. We
set    $\tilde{T}    =    T    \sqcup    \left(    \bigsqcup_{i\in    I}
  T_i\backslash\{\emptyset^{T_i}\} \right)  $ where the  symbol $\sqcup$
means  that   we  choose   for  the  sets    $(T_i)_{i\in  I}$
representatives of  isometry classes in $\T$ which  are disjoint subsets
of some common  set and that we perform the disjoint  union of all these
sets. We  set $\emptyset^{\tilde T}=\emptyset^T$. The set  $\tilde T$ is
endowed with the following metric $d^{\tilde T}$: if $s,t\in \tilde T$,
\begin{equation*} 
d^{\tilde T} (s,t) = 
\begin{cases} 
d^T(s,t)\ & \text{if}\ s,t\in T, \\
d^T(s,x_i)+d^{T_i}(\emptyset^{T_i},t)\ & \text{if}\
s\in T,\ t\in T_i\backslash\{\emptyset^{T_i}\} , \\ 
d^{T_i}(s,t)\ & \text{if}\ s,t\in T_i\backslash\{\emptyset^{T_i}\} ,\\
d^T(x_i,x_j)+d^{T_j}(\emptyset^{T_j},s)+d^{T_i}
(\emptyset^{T_i},t)\ 
& \text{if}\ i\neq j \ \text{and}\ s\in T_j\backslash\{\emptyset^{T_j}\} ,\ t\in
T_i\backslash\{\emptyset^{T_i}\} .
\end{cases} 
\end{equation*}
We define the mass measure on $\tilde T$ by:
\[
\mathbf{m}^{\tilde T}=\mathbf{m}^T+\sum_{i\in I}\left(\ind_{
 T_i\backslash\{\emptyset^{T_i}\}} \mathbf{m}^{T_i}+
\mathbf{m}^{T_i}(\{\emptyset^{T_i}\}) \delta_{x_i}\right),
\]
where $\delta_x$ is the Dirac mass at point $x$.  
 We will use the following notation:
\begin{equation} 
(\tilde T,d^{\tilde T},\emptyset^{\tilde T},
\mathbf{m}^{\tilde T} ) = T \circledast_{i\in I}(T_i,x_i) .
\label{def:gref}
\end{equation}
It is clear that the metric
space $(\tilde{T},d^{\tilde T},\emptyset^{\tilde T})$ is still a rooted complete
real tree.
Notice that it is not always true that $\tilde{T}$ remains locally
compact  or that
$\mathbf{m}^{\tilde T}$ defines a
locally finite measure on $\tilde T$. For instance if we consider the
grafting $\{\emptyset\}\circledast_{n\in\N}(T,\emptyset)$ where $T$ is
a non-trivial tree (i.e. we graft the same tree an infinite number of
times on a single point), then the resulting tree is not locally
compact.

It is easy to check that this grafting procedure is $\T$-compatible.

\subsection{Excursion measure of L\'evy tree}
\label{sec:loc.time}

Let $\psi$ be a critical or sub-critical branching mechanism defined by:
\begin{equation}
   \label{eq:psi}
\psi(\lambda)=\alpha\lambda+\beta\lambda^2
+\int_{(0,+\infty)}\left(\expp{-\lambda 
  r}-1+\lambda r\right)\pi(dr)
\end{equation}
with $\alpha\geq 0$, $\beta\ge 0$ and $\pi$ is a $\sigma$-finite measure
on      $(0,+\infty)$     such      that     $\int_{(0,+\infty)}(r\wedge
r^2)\pi(dr)<+\infty$ and $\langle \pi,1\rangle=+\infty$ if $\beta=0$. We also assume the Grey condition:
\begin{equation}
\int^{+\infty}\frac{d\lambda}{\psi(\lambda)}<+\infty.
\end{equation}
The  Grey  condition  is  equivalent  to  the  a.s.  finiteness  of  the
extinction time of the CSBP. This assumption is used to ensure that
the corresponding Lévy tree is compact. Let $v$ be the unique
non-negative solution of the equation:
\[
\forall a>0,\qquad \int_{v(a)}^{+\infty}\frac{d\lambda}{\psi(\lambda)}=a.
\]

We gather here results from \cite{dlg:pfalt}, Theorem 4.2, Theorem
4.3, Theorem 4.6, Theorem 4.7. Remarks of pages
575 and 578 of \cite{dlg:pfalt} state that the local time $\ell^a$ is
a function of the tree (see the third property below) and hence can be
defined on $\T$.

Using the coding of compact real trees by height functions, we can define  a    $\sigma$-finite   measure
$\N^\psi[d\ct]$ on $\T$, or  excursion measure  of Lévy tree,    with the
following properties.
\begin{enumerate}
   \item[(i)] \textbf{Height.} Recall Definition \reff{eq:def-hmax} of
     the height $H_{\text{max}}(\ct)$ of a tree. For all $a>0$,
     $$\N^\psi[H_{\text{max}}(\ct)>a]=v(a).$$

\item[(ii)] \textbf{Mass measure.} The mass measure $\bm^\ct$ is supported by
$\mathrm{Lf}(\ct)$, $\N^\psi[d\ct]$-a.e.

   \item[(iii)] \textbf{Local time.}
There exists a process $(\ell^a, a\geq 0)$ with values on finite
measures on $\ct$, which is
càdlàg for the weak topology on finite measures on $\ct$ and such that
$\N^\psi[d\ct]$-a.e.:
\begin{equation} 
\label{eq:int-la}
\mathbf{m}^{\ct}(dx) = \int_0^\infty \ell^a(dx) \, da,
\end{equation} 
$\ell^0=0$, $\inf\{a > 0 ; \ell^a = 0\}=\sup\{a \geq 0 ; \ell^a\neq
0\}=H_{\text{max}}(\ct)$ and for every fixed $a\ge 0$,
$\N^\psi[d\ct]$-a.e.: 
 \begin{itemize}
 \item The measure $\ell^a$ is supported on
$\ct(a)$. 
 \item We have for every bounded
continuous function $\phi$ on $\ct$:
\begin{align*} 
\langle\ell^a,\phi \rangle 
& = \lim_{\epsilon \downarrow 0}
\frac{1}{v(\epsilon)} \int \phi(x) \ind_{\{H_{\text{max}}(\ct')\ge
\epsilon\}} \cn_a^{\ct}(dx, d\ct') \\
 & = \lim_{\epsilon \downarrow 0} \frac{1}{v(\epsilon)} \int \phi(x)
\ind_{\{H_{\text{max}}(\ct')\ge \epsilon\}}
\cn_{a-\epsilon}^{\ct}(dx, d\ct'),\ \text{if}\ a>0.
\end{align*}
 \end{itemize}
 Under $\N^\psi$,  the real  valued process $(\langle\ell^a,1  \rangle ,
 a\geq  0)$ is  distributed as  a CSBP  with branching  mechanism $\psi$
 under its canonical measure.
\item[(iv)]  \textbf{Branching property.}
For every $a>0$, the conditional
distribution of the point measure $\cn_a^{\ct}(dx,d\ct')$ under
$\N^\psi[d\ct|H_{\text{max}}(\ct)>a]$, given $\pi_a(\ct)$, is that of a Poisson
point measure on $\ct(a)\times \T$ with intensity
$\ell^a(dx)\N^\psi[d\ct']$.
\item[(v)] \textbf{Branching points.}
\begin{itemize}
\item $\N^\psi[d\ct]$-a.e., the branching points of $\ct$ are of degree 3
  or $+\infty$.
\item The set of binary branching points (i.e. of degree 3) is empty
  $\N^\psi$ a.e if $\beta=0$ and is a countable dense subset of $\ct$
  if $\beta>0$.
\item The set $\mathrm{Br}_\infty(\ct)$  of infinite branching points is
  nonempty with $\N^\psi$-positive measure if and only if $\pi\ne 0$. If
  $\langle  \pi,1\rangle=+\infty$, the set  $\mathrm{Br}_\infty(\ct)$ is
  $\N^\psi$-a.e. a countable dense subset of $\ct$.
\end{itemize}
\item[(vi)] \textbf{Mass of the nodes.}
The set $\{ d(\emptyset,x),\ x\in \mathrm{Br}_\infty(\ct) \}$ coincides
$\N^\psi$-a.e. with the set of discontinuity times of the mapping $a\mapsto
\ell^a$. Moreover, $\N^\psi$-a.e., for every such discontinuity time $b$, there
is a unique $x_b\in \mathrm{Br}_\infty(\ct)\cap\ct(b)$ and
$\Delta_b>0$, such that:
\[ 
\ell^b = \ell^{b-} + \Delta_b \delta_{x_b}, 
\]
where $\Delta_b>0$ is called the mass of the node $x_b$. Furthermore
$\Delta_b$ can  be obtained
by the approximation: 
\begin{equation} 
\Delta_b = \lim_{\epsilon \rightarrow 0}
\frac{1}{v(\epsilon)}
n(x_b,\epsilon), 
\label{DefMas} 
\end{equation}
where $n(x_b,\epsilon)=\int \ind_{\{x=x_b\}}\ind_{\{H_{max}(\ct') >
\epsilon\}} \cn_b^\ct(dx,d\ct')$ is the number of sub-trees originating
from $x_b$ with height larger than $\epsilon$. 
\end{enumerate}

In order to stress the dependence in $\ct$, we may write $\ell^{a, \ct}$
for $\ell^a$. 

We set $\sigma^\ct$  or simply $\sigma$ when there  is no confusion, the
total mass of the mass measure on $\ct$:
\begin{equation} 
\label{eq:s=la}
\sigma=\bm^{\ct}(\ct).
\end{equation}
In particular, as $\sigma$ is distributed as the total mass of a CSBP
under its canonical measure, we have that $\N^\psi$-a.s. $\sigma>0$
and for $q> 0$ (see for instance \cite{k:ilflpa}, Corollary 10.9 for
the first equality, the others being obtained by differentiation):
\begin{equation}
   \label{eq:N1-es}
\N^\psi\left[1-\expp{-\psi(q) \sigma} \right]=q,\quad
 \N^\psi\left[\sigma\expp{-\psi(q)\sigma}\right]=\frac{1}{\psi'(q)} 
\quad \text{and}\quad 
\N^\psi\left[\sigma^2\expp{-\psi(q)\sigma}\right]
=\frac{\psi''(q)}{\psi'(q)^3}.   
\end{equation}
The last two equations hold for $q=0$ if $\psi'(0)>0$.
\subsection{Other measures on $\T$}

For each $r>0$, we  define  a probability  measure  $\P_r^\psi$ on $\T$  as  follows.  Let $r>0$  and
$\sum_{k\in\ck}\delta_{\ct^k}$ be  a Poisson point measure  on $\T$ with
intensity  $r\N^\psi$.  Consider  $\{\emptyset\}$  as the  trivial  w-tree reduced  to the  root with  null mass  measure. Define
$\ct=\{\emptyset\}   \circledast_{k\in  \ck}(   \ct^k,   \emptyset)$.  Using
Property (i) as well as \reff{eq:N1-es},  one easily gets that for
every $\varepsilon>0$ there is only a finite number of trees $\ct^k$
with height larger than $\varepsilon$. As each tree $\ct^k$ is
compact, we deduce that $\ct$ is a compact w-tree, and hence belongs
to $\T$. We denote by
$\P^\psi_r$ its distribution. Its corresponding local time is defined by
$\ell^a=\sum_{k\in \ck}  \ell^{a, \ct^k}$ and its total  mass is defined
by $\sigma=\sum_{k\in \ck}  \sigma^{\ct^k}$. Under $\P^\psi_r$, the real
valued process $(\langle\ell^a,1 \rangle , a\geq 0)$ is distributed as a
CSBP with branching mechanism $\psi$ with initial value $r$.\\

We consider the following measure on $\T$:
\begin{equation}
\label{eq:def-bN}
\bN^\psi[d\ct]=2\beta \N^\psi[d\ct]+\int_0^{+\infty}r\pi(dr)\,\P_r^\psi(d\ct)
\end{equation}
which appears as the grafting intensity in the tree-valued Markov
process of \cite{adh:etiltvp}.
From  \reff{eq:N1-es} and \reff{eq:def-bN}, elementary computations yield for $q>0$:
\begin{equation}
   \label{eq:bN-moment}
\bN^\psi\left[1-\expp{-\psi(q) \sigma}\right]=\psi'(q)
- \psi'(0), 
\end{equation}
as well as 
\begin{equation}
   \label{eq:bN-moment-1}
   \bN^\psi\left[\sigma\expp{-\psi(q) \sigma}
   \right]=\frac{\psi''(q)}{\psi'(q)} 
   \quad\text{and}\quad
   \bN^\psi\left[\sigma^2\expp{-\psi(q)
       \sigma}\right]=\inv{\psi'(q)} \partial_q
   \left(\frac{-\psi''(q)}{\psi'(q)} \right).
\end{equation}
The last two equalities also hold for $q=0$ if $\psi'(0)>0$. 

\subsection{Bismut decomposition of a L\'evy tree}
\label{subsec:Bismut}

We first present a decomposition of $T\in \TT$ according to a given
vertex $x\in T$. We denote by $(T^{j,\circ},j\in J_x)$
the connected components of $T\setminus\lb\emptyset,x\rb$. For every
$j\in J_x$, let $x_j$ be the MRCA of $T^{j,\circ}$ and consider
$T^j=T^{j,\circ}\cup  \{x_j\}$ as an element of $\T$ with mass
measure the mass measure of $T$ restricted to $T^{j,\circ}$. In order to graft
together all the sub-trees with the same MRCA, we consider the following
equivalence relation on $J_x$:
\[
j\sim j'\iff x_j=x_{j'}.
\]
Let $I_x^B$ be  the set of equivalence classes. For $[i]\in I_x^B$, we set
$x_{[i]}$ for the common value of $x_j$ with $j\in [i]$. We consider $\{x_{[i]}\}$ as
an element of $\T$ with mass measure $\bm^T(\{x_{[i]}\})\delta_{x_{[i]}}$. For
$[i]\in I_x^B$, we consider the following element of $\T$ defined by:
\[
T^{B,[i]} = \{x_{[i]}\} \circledast_{j\in [i]}(T^j,x_{[i]}) .
\]
Let  $h_{[i]}=d(\emptyset,  x_{[i]})$.  We  consider the  random  point  measure
$\cm_x^T$ on $\R_+\times\T$ defined by:
\[
   \cm_x^T=\sum_{[i]\in I_x^B}\delta_{(h_{[i]}, T^{B,[i]})}.
\]

Under  $\N^\psi$, conditionally  on  $\ct$, let  $U$  be a  $\ct$-valued
random  variable,  with   distribution  $\sigma^{-1}  \,  \bm^\ct$. In
other words, conditionally  on  $\ct$, $U$ represents a leaf chosen
``uniformly'' at random according to the mass measure $\bm^\ct$. We
define under $\N^\psi$ a  non-negative  random  variable  and
a random point measure on $\R_+\times\T$ as follows: 
\begin{equation}
   \label{eq:def-H-ZB}
H=d^\ct(\emptyset^\ct,  U)
\quad\text{and}\quad
\cz^B=\cm_U^\ct. 
\end{equation}
Let us remark that the distribution of $(H,\cz^B)$
does not depend on the choice of the representative in the equivalence
class and thus this random variable is well defined under $\N^\psi$.

By construction, for every non-negative measurable function $\Phi$ on
$\R_+\times\T$ and for every $\lambda\geq 0$, $\rho\ge 0$, we have:
\[
\N^\psi\left[\sigma \expp{-\lambda \sigma - \rho H - \langle \cz^B, \Phi
    \rangle}\right] =
\N^\psi\left[\int_\ct \bm^\ct(dx)\, \expp{-\lambda \sigma - \rho
    h(x) - \langle \cm_x^ \ct, \Phi
    \rangle}
\right]. 
\]

As a
direct consequence of Theorem 4.5 of \cite{dlg:pfalt}, we get the
following result. 

\begin{theo}\label{theo:Bismut}
For every non-negative measurable function $\Phi$ on
$\R_+\times\T$ and for every $\lambda\geq 0$, $\rho\ge 0$, we have:
\[
\N^\psi\left[\sigma \expp{-\lambda \sigma - \rho H - \langle \cz^B, \Phi
    \rangle}\right] 
=\int_0^{+\infty}da \expp{-\rho 
    a}\exp\left(-\int_0^ag(\lambda,u)du\right),
\]
where
\begin{equation}
   \label{eq:def-g}
g(\lambda,u)=\psi'(0)+\bN^\psi\left[1-\expp{-\lambda\sigma-\Phi(u,\ct)}\right].
\end{equation}
\end{theo}

In  other  words, under  $\N^\psi\left[\sigma,  \,  d\ct\right]$, if  we
choose  a leaf  $U$ uniformly   (i.e. according  to the normalized mass measure
$\bm^\ct$), the  height $H$ of this  leaf is distributed  according to the
density  $da\expp{-\psi'(0) a}$  and,  conditionally on  $H$, the  point
measure   $\cz^B$   is   a   Poisson   point process on $[0,H]$ with   intensity
$ \bN^\psi[d\ct]$.

\subsection{Pruning a L\'evy tree}
\label{sec:pruning}

A general pruning of a L\'evy tree has been defined in
\cite{adv:plcrt}. We use a special case of this pruning depending on a
one-dimensional parameter $\theta$ used first in \cite{v:dmfglt} to
define a fragmentation process of the tree.

More  precisely, under $\N^\psi[d\ct]$, we  consider  a mark  process
$M^\ct(d\theta,dy)$  on the  tree which  is a  Poisson point  measure on
$\R_+\times \ct$ with intensity:
\[
\ind_{[0,+\infty)}(\theta)d\theta\left(2\beta \ell^\ct(dy)+\sum_{x\in
   \mathrm{Br}_\infty(\ct)}\Delta_x\delta_x(dy)\right).
\]
The atoms $(\theta_i,y_i)_{i\in I}$ of this measure can be seen as
marks that arrive on the tree, $y_i$ being the location of the mark
and $\theta_i$ the ``time'' at which it appears. There are two kinds
of marks: some are ``uniformly'' distributed on the skeleton of the
tree (they correspond to the term $2\beta \ell^\ct$ in the intensity)
whereas the others are located on the infinite branching points of the tree, an infinite
branching point $y$ being first marked after an exponential time with
parameter $\Delta_y$.

For every $x\in\ct$, we set:
\[
\theta(x)=\inf\{\theta>0,\ M^\ct([0,\theta]\times \lb\emptyset,x\rb)>0\}.
\]
The process $(\theta(x),x\in\ct)$  is  called  the  record   process
on  the  tree  $\ct$ as  defined  in
\cite{ad:rpcrt}.  This corresponds  to the  first  time at  which a  mark
arrives on  $\lb\emptyset,x\rb$.  Using  this record process,  we define
the pruned tree at time $q$ as:
\[
\ct_q=\{x\in\ct,\ \theta(x)\ge q\}
\]
with  the  induced  metric,   root  $\emptyset$  and  mass  measure  the
restriction of the mass measure $\bm^\ct$.  If one cuts the tree $\ct$ at
time $\theta_i$  at point  $y_i$, then $\ct_q$  is the sub-tree  of $\ct$
containing the root at time $q$. Here again, the definition of $\ct_q$
is $\T$-compatible.

\begin{prop}\label{prop:law_pruned}(\cite{adv:plcrt}, Theorem 1.1)
For $q>0$ fixed, the distribution of $\ct_q$ under $\N^\psi$
is $\N^{\psi_q}$ with the branching mechanism $\psi_q$ defined for
$\lambda\geq 0$ by:
\begin{equation}
\label{eq:def-psi-theta}
\psi_q (\lambda)=\psi(\lambda+q)-\psi(q).
\end{equation}
\end{prop}

Furthermore, the measure $\N^{\psi_q}$ is absolutely continuous with
respect to $\N^\psi$, see \cite{ad:ctvmp}, Lemma 6.2: for every
$q\ge 0$ and every non-negative measurable function $F$ on $\T$, we have
\begin{equation}
\label{eq:Girsanov}
\N^{\psi_q}[F(\ct)]=\N^\psi\left[F(\ct)\expp{-\psi(q)\sigma}\right].
\end{equation}
We shall refer to this equation as  the Girsanov transformation for Lévy trees
as it corresponds to the Girsanov transformation of the height process 
(which is Brownian) in the quadratic case $\pi(dr)=0$. This
transformation corresponds also to the 
Essher transformation for the underlying Lévy process used in
\cite{dlg:rtlpsbp} to define  the height process in the general case. 
We deduce  from definition \reff{eq:def-bN} of $\bN^\psi$,  that for any
measurable non-negative functionals $F$ and $q\geq 0$:
\begin{equation}
   \label{eq:Girsanov-bN}
\bN^{\psi_q}[F(\ct)]=\bN^{\psi}\left[F(\ct)\expp{-\psi(q)
  \sigma}\right].
\end{equation}

Making $q$ vary  allows us to define a  tree-valued process $(\ct_q,q\ge
0)$ which is a Markov process under $\N^\psi$, see \cite{ad:ctvmp} Lemma
5.3 stated for the family of  exploration processes  which codes for the
corresponding Lévy  trees. The
process $(\ct_q,q\ge 0)$ is  a non-increasing process (for the inclusion
of  trees),  and  is  c\`adl\`ag.   Its  one-dimensional  marginals  are
described  in Proposition  \ref{prop:law_pruned} whereas  its transition
probabilities are  given by the  so-called special Markov  property (see
\cite{adv:plcrt}  Theorem  4.2  or  \cite{ad:ctvmp} Theorem  5.6).   The
time-reversed  process is also  a Markov  process and  its infinitesimal
transitions are  described in  \cite{adh:etiltvp} using a  point process
whose definition we recall now. We set:
\[
\{\theta_i,i\in I^R\}
\]
the set of jumping times of the process $(\ct_\theta,\theta\ge
0)$. For every $i\in I^R$, we set
$ \ct^{i,\circ}=\ct_{\theta_i-}\setminus\ct_{\theta_i}$ and  denote by
$x_i$ the MRCA  of 
$ \ct^{i,\circ}$. For $i\in I^R$,  we set:
\[
 \ct^i= \ct^{i,\circ}\cup\{x_i\}
\]
which is a real tree with distance the induced distance,  root $x_i$ and
mass measure the restriction of $\bm^\ct$ to $\ct^i$. Finally, we define,
conditionally on $\ct_0$, the
following random point measure on $\ct_0\times \T\times \R_+$:
\[
\cn=\sum_{i\in I^R}\delta_{(x_i, \ct^i,\theta_i)}.
\]

\begin{theo}[\cite{adh:etiltvp}, Theorem 3.2 and
  Lemma 3.3] 
\label{theo:Hoscheit}
Under  $\N^\psi$,  the predictable  compensator  of  the backward  point
process defined on $\R_+$ by:
\[
\theta\mapsto \ind_{\{\theta\le q'\}}\cn(dx,d\ct,dq')
\]
with respect to the backward left-continuous filtration
$\bF=(\cf_\theta,\theta\ge 0)$ defined by:
\[
\cf_\theta=\sigma((x_i, \ct^i,\theta_i),\ i\in I^R, \theta_i\ge
\theta)=\sigma(\ct_{q-},\ q\ge \theta).
\]
is given by:
\[
\mu(dx,d\ct,dq)=m^{\ct_q}(dx)\bN^{\psi_q}[d\ct]\ind_{\{q\ge 0\}}dq.
\]

And for any non-negative predictable process $\phi$ with
respect to the backward filtration $\bF$, we have:
\[
\N^\psi\left[\int\cn(dx,d\ct,dq)\phi(q,\ct_q,\ct_{q-})\right]
=\N^\psi\left[\int\mu(dx,dT,
  dq)\phi(q,\ct_q,\ct_q\circledast (T,x))\right].
\]
\end{theo}

\section{Statement of the main result}
\label{sec:result}

We keep the notations of the previous section. 
First notice that for $i\in I^R$, $\theta(x)=\theta_i$ for every $x\in\ct^i$. 
We set
$\sigma^i=m^{\ct}(\ct^i)=\sigma_{\theta_i-}-\sigma_{\theta_i}$ and
$\sigma_q=m^{\ct}(\ct_q)$ the total mass of $\ct_q$.  By
construction, we have for every $q\ge 0$:
\[
\sigma_q=\sum_{i\in I^R} \ind_{\{\theta_i\geq q\}} \sigma^i .
\]

We set:
\[
\Theta_q=\int_{\ct_q}(\theta(x)-q)\, m^\ct(dx).
\]
The quantity $\Theta:=\Theta_0$ appears in \cite{ad:rpcrt} as the limit of 
the number of cuts on the Aldous' CRT to isolate the root. 
Since $\theta(x)$ is constant on $\ct^i$, we get:
\[
\Theta_q=\sum_{i\in I^R}\ind_{\{\theta_i\ge q\}}\, (\theta_i-q)\sigma^i
=\int_q^{+\infty}\sigma_r\, dr.
\]
For simplicity, we write $\Theta$ for $\Theta_0$ and $\sigma$ for $\sigma_0$.

We consider the random point measure $\cz^R$ on $\R_+\times\T$
defined by:
\begin{equation}
   \label{eq:def-cz}
\cz^R=\sum_{i\in I^R}\delta_{(\Theta_{\theta_i},\ct^i)}.
\end{equation}
Recall the definition of $H$ and $\cz^B$  of Subsection \ref{subsec:Bismut}.

The main result of the paper is the next theorem that identifies the
law of the pair $(H, \cz^B)$ and the pair $(\Theta,\cz)$. 

\begin{theo}\label{theo:main}
  Assume  the Grey  condition  holds. For  every non-negative  measurable
  function  $\Phi$ on $\R_+\times\T$,  and every  $\lambda>0$, $\rho\geq
  0$, we have:
\[
\N^\psi\left[\sigma \expp{-\lambda \sigma - \rho H - \langle \cz^B, \Phi
    \rangle}\right]  =\N^\psi\left[\sigma\expp{-\lambda  \sigma  -  \rho
    \Theta -\langle\cz^R,\Phi\rangle}\right].
\]
\end{theo}

In particular $\Theta$ is distributed as the height $H$ of a leaf chosen
according to the  normalized mass measure on the  Lévy tree. 

Recall that
$\lim_{\varepsilon\rightarrow 0  } \N^\psi[\sigma>\varepsilon]=+\infty $
and $\lim_{\varepsilon  \rightarrow 0 } \N^\psi[\sigma\ind_{\{\sigma\leq
  \varepsilon \}}]=0$, as well as:
\[
\lim_{\varepsilon \rightarrow 0 }
\inv{\varepsilon}\N^\psi[\sigma\ind_{\{\sigma\leq \varepsilon \}}]=+\infty 
\] 
thanks to Lemma  4.1 from \cite{d:fhalp} (which is  stated for $\beta=0$
but which  also holds  for $\beta>0$).  The  next corollary is  a direct
consequence  of Theorem  \ref{theo:main} and  the properties  of Poisson
point  measures for  the Bismut  decomposition (see  Proposition  4.2 in
\cite{d:fhalp} for a proof of similar results).

\begin{cor}
   \label{cor:asymp-ti}
 Assume  the  Grey  condition  holds.    $\N^\psi$-a.e., we have:
\[
\lim_{\varepsilon\to 0}\frac{1}{\N^\psi[\sigma>\varepsilon]}
\sum_{i\in I^R}\ind_{\{\sigma^i\geq 
  \varepsilon\}}=\Theta.
\]
$\N^\psi$-a.e., for any positive sequence $(\varepsilon_n, n\geq 0)$
 converging  to $0$, there exists a subsequence $(\varepsilon_{n_k},
 k\geq 0)$ such that:
\[
\lim_{k\to+\infty}\inv{\N^\psi[\sigma\ind_{\{\sigma\leq \varepsilon_{n_k}\}}]}
\sum_{i\in I^R}\sigma^i\ind_{\{\sigma^i\le  \varepsilon_{n_k}\}}= \Theta.
\]

When  $\psi$ is  regularly  varying at  infinity  with index  $\gamma\in
(1,2]$, the previous convergence holds
$\N^\psi$-a.e.
\end{cor}

\section{Proof of the main result}
\label{sec:proof}

\subsection{Preliminaries results}

We first state a basic lemma.

\begin{lem}\label{lem:sumJ1}
Let $\cn_1=\sum_{j\in J_1}\delta_{r_j,x_j}$ be a point measure on
$[0,+\infty)$. If $\sum_{j\in J_1}x_j<+\infty$, then for every $r\ge
  0$, we have:
\begin{equation}\label{eq:sumJ1}
1-\exp\left(-\sum_{j\in J_1}\ind_{\{r_j\ge r\}}x_j\right)=\sum_{j\in
    J_1}\ind_{\{r_j\ge r\}}(1-\expp{-x_j})\exp\left(-\sum_{\ell\in
    J_1}\ind_{\{r_\ell >r_j\}}x_\ell\right).
\end{equation}
\end{lem}

\begin{proof}
The result is obvious for $J_1$ finite. For the infinite case,  for
$\varepsilon>0$ consider the finite set:
\[
J_{1,\varepsilon}=\{j\in J_1,\ x_j\ge \varepsilon\}.
\]
Apply Formula \reff{eq:sumJ1} with $J_1$ replaced by
$J_{1,\varepsilon}$ and then conclude by letting $\varepsilon$ tend to
  0 thanks to monotone convergence and dominated convergence.
\end{proof}

Since $\ct_q$ is distributed  according to $\N^{\psi_q}$, we deduce from
\reff{eq:N1-es} together with \reff{eq:Girsanov-bN} that for $q>0$:
\begin{equation}
   \label{eq:Ns1}
\N^\psi[\sigma_q]
=\N^{\psi_q}[\sigma]=\frac{1}{\psi'(q)}, 
\quad 
\N^\psi[\sigma_q^2]
=\N^{\psi_q}[\sigma^2] =\frac{\psi''(q)}{\psi'(q)^3}
\cdot
\end{equation}

\subsection{Laplace transform of $(\sigma,\Theta,\cz^R)$}

\begin{prop}\label{prop:laplace-sig-the-N}
  Let    $\Phi$    be   a    non-negative    measurable   function    on
  $\R_+\times\T$.  Assume that
  $\langle\cz^R,\Phi\rangle   <+\infty$   $\N^\psi$-a.e.   and   for   all
  $\lambda>0$, $\sup_{u\geq 0} g(\lambda,  u)<+\infty $ with $g$ defined
  by \reff{eq:def-g}. Then, for all  $\lambda> 0$ and $\rho\ge 0$,
  we have:
\begin{equation}
   \label{eq:gq=1}
\N^\psi\left[\sigma\bigl(\rho+g(\lambda,\Theta)\bigr)
  \expp{-\lambda\sigma-\rho\Theta-\langle\cz^R,\Phi\rangle}\right]=1. 
\end{equation}
\end{prop}

\begin{proof}
For every $\varepsilon>0$, $q\geq 0$, we set:
\[
\sigma_q^{\varepsilon}
=\sum_{i\in I^R}\ind_{\{\theta_i\ge
  q\}}\ind_{\{\sigma^i\geq \varepsilon\}}\sigma^i, 
\qquad 
\Theta_q^{\varepsilon}
=\sum_{i\in I^R }\ind_{\{\theta_i\ge
  q\}}\ind_{\{\sigma^i\geq \varepsilon\}} \sigma^i(\theta^i-q),
\]
and
\[
Z_q^{\varepsilon}
=\sum_{i\in I^R}\ind_{\{\theta_i\ge
  q\}}\ind_{\{\sigma^i\geq \varepsilon\}}
\Phi(\Theta_{\theta_i},\ct^i), 
\qquad
Z_q=\sum_{i\in I^R}\ind_{\{\theta_i\ge
  q\}}\Phi(\Theta_{\theta_i},\ct^i),
\]
so that $Z_0= \langle \cz^R, \Phi \rangle$. 
For every $\varepsilon>0$, $q>0$, we set: 
\[
\varphi^\varepsilon_q(\lambda,\rho)
=\N^\psi\left[1-\exp(-\lambda\sigma_q^{\varepsilon}-
  \rho\Theta_q^{\varepsilon}-Z_q^{\varepsilon})\right]. 
\]
Since    $\langle\cz^R,\Phi \rangle$    is    finite by assumption,   we    get    that
$Z_q^{\varepsilon}$ is finite. We use Lemma \ref{lem:sumJ1} to get:
\begin{multline*}
\varphi^\varepsilon_q(\lambda,\rho)
=\N^\psi\left[\sum_{i\in I^R }\ind_{\{ \theta_i\ge
    q\}}\ind_{\{\sigma^i\ge     \varepsilon\}}
  \biggl(1-\exp\Bigl(-\bigl(\lambda+\rho(\theta_i-q)\bigr) 
  \sigma^i-\Phi(\Theta_{\theta_i},\ct^i)\Bigr)  \biggr)\right.\\
\left(\exp\left(-\sum_{\ell\in I^R }\ind_{\{
      \theta_\ell>\theta_i\}} \ind_{\{\sigma^\ell\ge \varepsilon\}}
    \Bigl(\bigl( \lambda+\rho(\theta_\ell-q)\bigr) \sigma^\ell
    +\Phi(\Theta_{\theta_\ell},\ct^\ell) \Bigr) \right)\right]. 
\end{multline*}
Then,   if    we   use   Theorem    \ref{theo:Hoscheit}   (recall   that
$\sigma_q=m^{\ct_q}(\ct_q)$), we get:
\[
\varphi^\varepsilon_q(\lambda,\rho)
=\N^\psi\left[\int_q^{+\infty }dr\,
  \sigma_r\,
G_r^\varepsilon(\lambda+\rho(r-q),\Theta_r)\, \,
  \exp\Bigl(-\bigl(\lambda+\rho(r-q)\bigr)\sigma_r^{\varepsilon}
-\rho\Theta_r^{\varepsilon} -Z_r^{\varepsilon}
\Bigr)\right], 
\]
with
\[
G_r^\varepsilon(\kappa,t)=
\bN^{\psi_r}\left[\ind_{\{\sigma\ge\varepsilon\}}  \left(1-\expp{-
 \kappa\sigma-\Phi(t,\ct)}\right)\right].
\]
Thanks to \reff{eq:bN-moment-1} and \reff{eq:Girsanov-bN}, we get:
\begin{equation}
   \label{eq:majoG}
0\leq G_r^\varepsilon(\kappa,t)\leq
\bN^{\psi_r}\left[\ind_{\{\sigma\ge\varepsilon\}}  
\right] \leq  \inv{\varepsilon} \bN^{\psi_r}\left[\sigma\right] 
=\inv{\varepsilon}\frac{\psi''(r)}{\psi'(r)}\cdot
\end{equation}
Since $\psi''$  is non-increasing and $\psi'$ is  non-decreasing, we get
that    for    fixed     $q>0$,    the    map    $r\mapsto    \partial_r
\left(\frac{-\psi''(r)}{\psi'(r)}  \right)$ is non-negative  and bounded
for $r>q$.  We deduce from \reff{eq:bN-moment-1} and
\reff{eq:Girsanov-bN} that:
\[
 \bN^{\psi_r}\left[\ind_{\{\sigma\ge\varepsilon\}}  \sigma \expp{-
 \kappa\sigma-\Phi(t,\ct)}\right]
\leq  \inv{\varepsilon} \bN^{\psi_r}\left[\sigma^2\right] 
=\inv{\varepsilon}
\inv{\psi'(r)} \partial_r
   \left(\frac{-\psi''(r)}{\psi'(r)} \right).
\]
We deduce that the  map $\kappa  \mapsto  G_r^\varepsilon(\kappa,t)$  is
$\cc^1$   and:
\begin{equation}
   \label{eq:deriv-G}
0\leq  \partial_\kappa
G_r^\varepsilon(\kappa,t)
=\bN^{\psi_r}\left[\ind_{\{\sigma\ge\varepsilon\}}  \sigma \expp{-
 \kappa\sigma-\Phi(t,\ct)}\right]
\leq  \inv{\varepsilon}
\inv{\psi'(r)} \partial_r
   \left(\frac{-\psi''(r)}{\psi'(r)} \right). 
\end{equation}

We set:
\[
H^\varepsilon_{r,\lambda}(q)=\N^\psi\left[
  \sigma_r\,
G_r^\varepsilon(\lambda+\rho(r-q),\Theta_r)\, \,
  \exp\Bigl(-\bigl(\lambda+\rho(r-q)\bigr)\sigma_r^{\varepsilon}
-\rho\Theta_r^{\varepsilon}-Z_r^{\varepsilon}  \Bigr)\right], 
\]
so that:
\[
\varphi^\varepsilon_q(\lambda,\rho)= \int_q^{+\infty} H^\varepsilon_{r,\lambda}(q)\, dr.
\]
Thanks to \reff{eq:majoG} and \reff{eq:Ns1}, we get $ 0\leq H^\varepsilon_{r,\lambda}(q)\leq
\varepsilon^{-1} \psi''(r)/\psi'(r)^2$. This implies in turn that
$\varphi^\varepsilon_q(\lambda, \rho)\leq \varepsilon^{-1}/ \psi'(q)$. 

For $r>0$, $\kappa>0$,  we set:
\[
h^\varepsilon_{r}(\kappa)=\N^\psi\left[
  \sigma_r\, \left(\partial_\kappa
G_r^\varepsilon(\kappa,\Theta_r)
+\sigma_r^\varepsilon G_r^\varepsilon(\kappa,\Theta_r)
\right)\,   \expp{-\kappa\sigma_r^{\varepsilon}
-\rho\Theta_r^{\varepsilon}-Z_r^{\varepsilon} }\right]. 
\]
Since $\sigma_
r^{\varepsilon}\leq \sigma_r$, we have, using
\reff{eq:Ns1}:
\begin{align*}
0\leq h^\varepsilon_{r}(\kappa)
&\leq  \inv{\varepsilon}
\N^\psi\left[
  \sigma_r \inv{\psi'(r)} 
\partial_r
   \left(\frac{-\psi''(r)}{\psi'(r)} \right)  + \sigma_r^2
   \frac{\psi''(r)}{\psi'(r)} \right] \\ 
&\leq  \inv{\varepsilon} \left[
\inv{\psi'(r)^2} 
\partial_r
   \left(\frac{-\psi''(r)}{\psi'(r)} \right)  + 
   \frac{\psi''(r)^2}{\psi'(r)^4} \right] .
\end{align*}
By monotonicity, we get:
\begin{multline*}
\int_{[q,+\infty )^2}  duds\, \ind_{\{u<s\}}
h^\varepsilon_{s}(\lambda+\rho(s-u)) \\
\begin{aligned}
&\leq  \int_{[q,+\infty )^2}  duds\, \ind_{\{u<s\}} 
\inv{\varepsilon}\left[
\inv{\psi'(s)^2} 
\partial_s
   \left(\frac{-\psi''(s)}{\psi'(s)} \right)  + 
   \frac{\psi''(s)^2}{\psi'(s)^4} \right] \\
&\leq  \int_{[q,+\infty )^2}  duds\, 
\inv{\varepsilon}\ind_{\{u<s\}} \left[
\inv{\psi'(u)^2} 
\partial_s
   \left(\frac{-\psi''(s)}{\psi'(s)} \right)  + 
   \frac{\psi''(u)}{\psi'(u)^2} \frac{\psi''(s)}{\psi'(s)^2}\right] \\
&= \frac{2}{\varepsilon} \int_{[q,+\infty )}  du\, 
\frac{\psi''(u)}{\psi'(u)^3}\\
&= \frac{1}{\varepsilon}\inv{\psi'(q)^2}\cdot
\end{aligned}   
\end{multline*}
We deduce that 
 the maps $u\mapsto
H^\varepsilon_{s,\lambda}(u)$ and $\lambda \mapsto H^\varepsilon_{s,\lambda}(u)$ are $\cc^1$
for $\lambda\geq 0$, $s\geq u\geq q$, with:
\[
\partial_u  H^\varepsilon_{s,\lambda}(u)=-\rho \partial_\lambda H^\varepsilon_{s,\lambda}(u)
\quad\text{and}\quad  \left|\partial_\lambda H^\varepsilon_{s,\lambda}(u) \right| \leq
h^\varepsilon_{s}(\lambda+\rho(s-u)).
\]
Thus we have $\int_{[q,+\infty )^2}  duds\, \ind_{\{u<s\}} \left|\partial_u
  H^\varepsilon_{s,\lambda}(u) \right| \leq  \rho/\varepsilon \psi'(q)^2$. 
Then, elementary computation yields:
\[
\varphi^\varepsilon_q(\lambda,\rho)= \int_q^{+\infty} H^\varepsilon_{r,\lambda}(q)\, dr
= \int_q^{+\infty} du\, \left[ H^\varepsilon_{u, \lambda}(u) - \int_u^{+\infty} ds
  \, \partial_u H^\varepsilon_{s,
    \lambda}(u) \right].
\]
We deduce that the maps $q\mapsto \varphi^\varepsilon_q(\lambda,\rho)
$ and $\lambda \mapsto \varphi^\varepsilon_q(\lambda,\rho)$ are $\cc^1$ and:
\[
\partial_q \varphi^\varepsilon_q(\lambda,\rho)
= - H^\varepsilon_{q,\lambda} (q) +  \int_q^{+\infty} ds
  \, \partial_u H^\varepsilon_{s,
    \lambda}(q) 
=- H^\varepsilon_{q,\lambda} (q)  -\rho \partial_\lambda \int_q^{+\infty} ds
  \,  H^\varepsilon_{s,
    \lambda}(q). 
\]
With $H^\varepsilon_{q,\lambda}(q)=\N^\psi\left[
  \sigma_q\,
G_q^\varepsilon(\lambda,\Theta_q)\, \,
  \exp\Bigl(-\lambda\sigma_q^{\varepsilon}
-\rho\Theta_q^{\varepsilon}-Z_q^{\varepsilon}  \Bigr)\right]$, we deduce that:
\begin{equation}
 \label{eq:partial-q}
\partial_q \varphi^\varepsilon_q(\lambda,\rho)= - \N^\psi\left[
  \sigma_q\,
G_q^\varepsilon(\lambda,\Theta_q)\, \,
  \exp\Bigl(-\lambda\sigma_q^{\varepsilon}
-\rho\Theta_q^{\varepsilon}-Z_q^{\varepsilon}  \Bigr)\right] 
  -\rho \partial_\lambda
  \varphi^\varepsilon_q(\lambda,\rho). 
\end{equation}
We also have:
\begin{equation}
\label{eq:partial-lambda}
\partial_\lambda\varphi^\varepsilon_q(\lambda,\rho)=\N^\psi\left[\sigma_q^{\varepsilon}
  \exp(-\lambda \sigma_q^{\varepsilon} - 
  \rho \Theta_q^{\varepsilon} -Z_q^{\varepsilon} )\right]. 
\end{equation}
Moreover, thanks to Girsanov formula \reff{eq:Girsanov}, we have:
\[
\varphi^\varepsilon_q(\lambda,\rho) 
=\N^\psi\left[\left(1-\exp(-\lambda\sigma_0^{\varepsilon}
        -\rho\Theta_0^\varepsilon-Z_0^\varepsilon )\right)
  \expp{-\psi(q)\sigma}\right].
\]
We deduce that:
\begin{align*}
\partial_q \varphi^\varepsilon_q(\lambda,\rho) 
&=- \psi'(q) \N^\psi\left[\sigma\left(1-\exp(-\lambda\sigma_0^{\varepsilon}
        -\rho\Theta_0^\varepsilon-Z_0^\varepsilon )\right)
  \expp{-\psi(q)\sigma}\right]\\
&=- 1+ \psi'(q) \N^\psi\left[\sigma_q
  \, \exp(-\lambda\sigma_q^{\varepsilon} 
        -\rho\Theta_q^\varepsilon-Z_q^\varepsilon )
\right]. 
\end{align*}
We deduce from \reff{eq:partial-q} and \reff{eq:partial-lambda} that:
\begin{equation}
   \label{eq:sge=1}
\N^\psi\left[ \Big(\sigma_q (\psi' (q)+
    G^\varepsilon_q(\lambda,\Theta_q)) + \rho \sigma_q^\varepsilon
  \Big) \, \exp\left(-\lambda\sigma_q^{\varepsilon} 
        -\rho\Theta_q^\varepsilon-Z_q^\varepsilon \right)
\right]
=1.
\end{equation}
Using  Girsanov formula \reff{eq:Girsanov-bN} and
\reff{eq:bN-moment}, we get:
\[
G_q^\varepsilon (\lambda, t) \leq 
G^0_q(\lambda,t)=g(\lambda+\psi(q),t)-\psi'(0)- \bN^{\psi} \left[1
  -\expp{-\psi(q)\sigma}\right] 
= g(\lambda+\psi(q),t)-\psi'(q).
\]
We deduce that:
\[
\sigma_q (\psi' (q)+
    G^\varepsilon_q(\lambda,\Theta_q)) + \rho \sigma_q^\varepsilon
\leq \sigma_q (\sup_{t\geq 0} g(\lambda+\psi(q),t) +\rho). 
\]
By dominated convergence, letting $\varepsilon$ decrease to $0$ in
\reff{eq:sge=1}, we
deduce that:
\[
\N^\psi\left[ \sigma_q \Big(g(\lambda+\psi(q),\Theta)  +\rho 
  \Big) \, \exp\left(-\lambda\sigma_q
        -\rho\Theta_q-Z_q \right)
\right]
=1.
\]
Using  Girsanov formula \reff{eq:Girsanov} once again, we get:
\[
\N^\psi\left[ \sigma\Big(g(\lambda+\psi(q),\Theta)  +\rho 
  \Big) \, \exp\left(-(\lambda+\psi(q))\sigma
        -\rho\Theta-\langle \cz^R , \Phi \rangle \right)
\right]
=1.
\]
Since $\lambda>0$ and $q>0$ are arbitrary, we deduce that \reff{eq:gq=1}
holds. 
\end{proof}

We   deduce   the   following   corollary   which   states   that the
pair $H$
and the projection of $\cz^B$ on $\T$ have the same distribution as
$\Theta$ and the projection of $\cz^R$ on $\T$.

Let $\gamma$ be a non-negative measurable function defined  on $\T$. For
a measure $\cz$ on $\R_+\times \T$, we shall abuse notation and write:
\[
\langle \cz, \gamma \rangle=\int \gamma(T)\, \cz(dt,dT).
\]

\begin{cor}\label{cor:first-eq-law}
For every non-negative measurable function $\gamma$ on $\T$ such that
$\gamma(\ct)=0$ if $m^\ct(\ct)=0$,  and every
$\lambda\geq 0$, $\rho\ge 0$, we have:
\begin{equation}
\label{eq:h=q1}
\N^\psi\left[\sigma \expp{-\lambda \sigma - \rho H - \langle \cz^B, \gamma
    \rangle}\right] 
=\N^\psi\left[\sigma\expp{-\lambda \sigma -\rho \Theta
    -\langle\cz^R,\gamma\rangle}\right]. 
\end{equation}
\end{cor}

\begin{proof}
  Let  $\lambda>0$.   Recall  $\sigma=m^\ct(\ct)$.   First  assume  that
  $\gamma(\ct)\leq  c  \sigma$ for  some  finite  constant $c$.   Taking
  $\Phi(t,\ct)=\gamma(\ct)$  in  Theorem \ref{theo:Bismut} and  using  that
  $g(\lambda,u)$ doesn't depend on $u$, we get:
\[
\N^\psi\left[\sigma \expp{-\lambda \sigma - \rho H - \langle \cz^B, \gamma
    \rangle}\right] 
=\frac{1}{\rho+g(\lambda,0)} \cdot
\]
Notice that $\langle \cz^R, \Phi \rangle \leq  c\sigma$ and thus
hypotheses from Proposition \ref{prop:laplace-sig-the-N} are in force. 
We deduce from Proposition \ref{prop:laplace-sig-the-N} that:
\[
\N^\psi\left[\sigma\exp\left( -\lambda \sigma-\rho\Theta-\langle \cz^R, \gamma
    \rangle \right)\right]
=\frac{1}{\rho+g(\lambda,0)} \cdot
\]
Thus equality \reff{eq:h=q1} holds. Use monotone convergence to remove
hypotheses  $\lambda>0$ and $\gamma(\ct)\leq c \sigma$
for some finite constant $c$. 
\end{proof}

\subsection{Proof of Theorem \ref{theo:main}}

Let  $\Phi$ be  a measurable  non-negative function  defined on  the space
$\R_+\times  \T$. Let  us assume  that for  every  $\ct\in\T$, $t\mapsto
\Phi(t,\ct)$  is  continuous,   $\langle  \cz^R,\Phi\rangle$  is  finite
$\N^\psi$-a.s.  and that the  function $g$ defined by \reff{eq:def-g} is
bounded for any $\lambda>0$ as a function of $u$. We set:
\[
\Gamma^R (r,h)=\N^\psi\left[\expp{-\langle  \cz^R,\Phi\rangle}\bigm|
  \sigma=r,\ \Theta=h\right]. 
\]
We deduce from Proposition \ref{prop:laplace-sig-the-N} and Corollary
\ref{cor:first-eq-law} that for every $\lambda>0$, $\rho\ge 0$, we have:
\begin{align}
\nonumber
   1
&=\N^\psi\left[\sigma\bigl(\rho+g(\lambda,\Theta)\bigr)
  \expp{-\lambda\sigma-\rho\Theta-\langle\cz^R,\Phi\rangle}\right]\\
\nonumber
&=\N^\psi\left[\sigma\bigl(\rho+g(\lambda,\Theta)\bigr)
  \expp{-\lambda\sigma-\rho\Theta} \Gamma^R (\sigma,\Theta)\right]\\
\label{eq:lap}
&=\N^\psi\left[\sigma
  \Big(\rho+g(\lambda,H)\Big)\expp{-\lambda \sigma-\rho H}
  \Gamma^R(\sigma,H)\right]. 
\end{align}
Let  $ \sum_{i\in  I}\delta_{(h_i,\ct_i)}$  be  a  Poisson
measure  with  intensity $dh\,  \bN^\psi[d\ct]$  under some  probability
measure      $P$.      For     every      $i\in      I$,     we      set
$\sigma_i=m^{\ct_i}(\ct_i)$. Then for every $h>0$, we set:
\[
\sigma(h)=\sum_{i\in I}\ind_{\{h_i\le h\}}\sigma_i.
\]
Equation \reff{eq:lap} and Theorem \ref{theo:Bismut} imply that:
\[
\int_0^{+\infty}dh\, \expp{-(\rho+\psi'(0)) h}
\expp{-G(h)}(\rho+g(\lambda,h))=1,
\]
with:
\[
G(h)=-\log\left(E\left[\expp{-\lambda\sigma(h)}\Gamma^R(\sigma(h),h)\right]\right).
\]
We deduce that:
\[
\int_0^{+\infty}dh\, \expp{-\rho
  h}\left[1-\expp{-\psi'(0)h-G(h)}\right]=\int_0^{+\infty}
\inv{\rho} \expp{-\rho h} \, dA(h)
= \int_0^{+\infty}dh\,
\expp{-\rho h}A(h),
\]
with:
\[
A(h)=\int_0^hdu\, \expp{-\psi'(0)u-G(u)}g(\lambda,u).
\]
Since this holds for every $\rho\ge 0$, uniqueness of the Laplace transform implies
that:
\begin{equation}\label{eq:A}
A(h)=1-\expp{-\psi'(0)h-G(h)} \quad\text{a.e.}
\end{equation}
Since $A$ is  continuous, there exists a continuous  function $\tilde G$
such  that   a.e.  $\tilde   G=G$.  Since,  $t\mapsto   \Phi(t,\ct)$  is
continuous,   we  get   that,   for  every   $\lambda\ge  0$,   $u\mapsto
g(\lambda,u)$ is continuous. Then $A$ is of class $\mathcal{C}^1$ and so
is $\tilde G$. Moreover, by differentiating \reff{eq:A}, we get:
\[
\psi'(0)+\tilde G'(h)=g(\lambda,h).
\]
Since $A(0)=0$, we get  $\tilde G(0)=0$, and thus $
\psi'(0)h+\tilde G(h)=\int_0^hg(\lambda,u)du$. 
This implies that:
\begin{equation}
   \label{eq:Gg}
\int_0^hg(\lambda,u)du = G(h)+\psi'(0)h \quad \text{a.e.}
\end{equation}

We have:
\begin{align*}
\N^\psi\left[\sigma \expp{-\lambda \sigma-\rho H - \langle  \cz^B,\Phi\rangle}
  \right] 
& =\int_0^{+\infty}dh\,
\expp{-\rho h-\int_0^hg(\lambda,u)du}\\
& =\int_0^{+\infty}dh\,
\expp{-(\rho+\psi'(0))h-G(h)}\\
& =\int_0^{+\infty}dh\,
\expp{-(\rho+\psi'(0))h}E\left[\expp{-\lambda\sigma(h)}\Gamma^R(\sigma(h),h)\right]\\
& =\N^\psi\left[\sigma \expp{-\lambda\sigma-\rho
    H}\Gamma^R(\sigma,H)\right]\\
&=\N^\psi\left[\sigma \expp{-\lambda\sigma-\rho
\Theta }\Gamma^R(\sigma,\Theta)\right]\\
&=\N^\psi\left[\sigma \expp{-\lambda\sigma-\rho
\Theta  -\langle \cz^R, \Phi \rangle}\right],
\end{align*}
where  we  used  Theorem   \ref{theo:Bismut}  for  the  first  and  fourth
equalities, \reff{eq:Gg} for  the second, the definition of  $G$ for the
third, Corollary \ref{cor:first-eq-law}  (which states that $(\sigma,H)$
and $(\sigma,\Theta)$  have the  same distribution under  $\N^\psi$) for
the fifth, and the definition of $\Gamma^R$ for the last.

As $\N^\psi\left[\sigma\expp{-\lambda \sigma}\right]$  is finite, we can
remove   using    dominated   convergence   the    hypothesis   $\langle
\cz^R,\Phi\rangle$ finite.  The   function  $g$   defined  by
\reff{eq:def-g},      with      $\Phi(t,      \ct)$     replaced      by
$\Phi(t,\ct)\ind_{\{\sigma\leq 1/n\}}$,  is bounded for  any $\lambda>0$
as a  function of $u$. Thus,  using again dominated  convergence, we can
remove  the hypothesis  on  $\Phi$  such that  function  $g$ defined  by
\reff{eq:def-g} is  bounded for  any $\lambda>0$ as  a function  of $u$.
Then use monotone class theorem  to remove the continuity hypothesis on
$\Phi$ and end the proof.


\begin{thebibliography}{10}

\bibitem{ad:rpcrt}
R.~ABRAHAM and J.~DELMAS.
\newblock Record process on the continuum random tree.
\newblock arXiv:1107.3657, 2011.

\bibitem{ad:ctvmp}
R.~ABRAHAM and J.~DELMAS.
\newblock A continuum-tree-valued {M}arkov process.
\newblock {\em Ann. of Probab.}, 40:1167--1211, 2012.

\bibitem{adh:pgwttvmp}
R.~ABRAHAM, J.~DELMAS, and H.~HE.
\newblock Pruning {G}alton-{W}atson trees and tree-valued {M}arkov processes.
\newblock {\em Ann. Inst. H. Poincar\'e}, 48:688--705, 2012.

\bibitem{adh:etiltvp}
R.~ABRAHAM, J.~DELMAS, and P.~HOSCHEIT.
\newblock Exit times for an increasing {L}\'evy tree-valued process.
\newblock arXiv:1202.5463, 2012.

\bibitem{adh:ghptwms}
R.~ABRAHAM, J.~DELMAS, and P.~HOSCHEIT.
\newblock A note on Gromov-{H}ausdorff-{P}rohorov distance between
(locally) compact measure spaces.
\newblock arXiv:1202.5464, 2012.

\bibitem{adv:plcrt}
R.~ABRAHAM, J.~DELMAS, and G.~VOISIN.
\newblock Pruning a {L}\'evy continuum random tree.
\newblock {\em Elec. J. of Probab.}, 15:1429--1473, 2010.

\bibitem{abbh:cdtmc}
L.~ADDARIO-BERRY, N.~BROUTIN, and C.~HOLMGREN.
\newblock Cutting down trees with a {M}arkov chainsaw.
\newblock arXiv:1110.6455, 2012.

\bibitem{a:crt1}
D.~ALDOUS.
\newblock The continuum random tree {I}.
\newblock {\em Ann. of Probab.}, 19:1--28, 1991.

\bibitem{ap:sac}
D. ALDOUS and J. PITMAN
\newblock The standard additive coalescent.
\newblock {\em Ann. of Probab.}, 26:1703--1726, 1998.

\bibitem{ap:tvmcdgwp}
D. ALDOUS and J. PITMAN
\newblock Tree-valued {M}arkov chains derived from {G}alton-{W}atson
processes.
\newblock {\em Ann. Inst. H. Poincar\'e}, 34:637--686, 1998.

\bibitem{b:ft}
J.~BERTOIN.
\newblock Fires on trees.
\newblock arXiv:1011.2308, 2010.

\bibitem{bm:ctlgwtbcrt}
J.~BERTOIN and G.~MIERMONT.
\newblock The cut-tree of large {G}alton-{W}atson trees and the {B}rownian
  {CRT}.
\newblock arXiv:1201.4081, 2012.

\bibitem{c:ilt}
I. CHRISWELL.
\newblock {\em Introduction to $\Lambda$-trees.}
\newblock World Scientific Publishing, River Edge, NJ, 2001.

\bibitem{d:fhalp}
J.-F. DELMAS.
\newblock Fragmentation at height associated to {L}évy processes.
\newblock {\em Stoch. Proc. Appl.}, 117(3):297--311, 2007. 

\bibitem{dmt:tt}
A.W.M. DRESS, V. MOULTON and W.F. TERHALLE.
\newblock T-theory.
\newblock {\em European J. Combin.}, 17:161--175, 1996.

\bibitem{d:ltcpcgwt}
T.~DUQUESNE.
\newblock A limit theorem for the contour process of conditioned
Galton-Watson trees. 
\newblock{\em Ann. Probab.}, 31:996--1027, 2003.

\bibitem{dlg:rtlpsbp}
T.~DUQUESNE and J.~L. GALL.
\newblock {\em Random trees, {L}\'evy processes and spatial branching
  processes}, volume 281.
\newblock Ast\'erisque, 2002.

\bibitem{dlg:pfalt}
T.~DUQUESNE and J.~L. GALL.
\newblock Probabilistic and fractal aspects of {L}\'evy trees.
\newblock {\em Probab. Th. and rel. Fields}, 131:553--603, 2005.

\bibitem{e:prt}
S.~EVANS.
\newblock {\em Probability and real trees}, volume 1920 of {\em Ecole d'\'et\'e
  de Probabilit\'es de Saint-Flour, Lecture Notes in Math.}
\newblock Springer, 2008.

\bibitem{epw:rprtrgr}
S.~EVANS, J.~PITMAN, and A.~WINTER.
\newblock Rayleigh processes, real trees and root growth with re-grafting.
\newblock {\em Probab. Th. and rel. Fields}, 134:81--126, 2005.

\bibitem{ew:sprrrtmp}
S.~EVANS, and A.~WINTER.
\newblock Subtree prune and regraft: a reversible real tree-valued
Markov process. 
\newblock {\em Ann. Probab.}, 34(3):918--961,  2006.

\bibitem{lg:rta}
J.~L. GALL.
\newblock Random trees and applications.
\newblock {\em Probability Surveys}, 2:245--311, 2005.

\bibitem{lglj:bplpep}
J.~L. GALL and Y.~L. JAN.
\newblock Branching processes in {L}\'evy processes: the exploration process.
\newblock {\em Ann. of Probab.}, 26:213--252, 1998.

\bibitem{gpw:cdrmmslcmt}
A.~GREVEN, P.~PFAFFELHUBER, and A.~WINTER.
\newblock Convergence in distribution of random metric measure spaces
  ($\Lambda$-coalescent measure trees).
\newblock {\em Probab. Th. and rel. Fields}, 145:285--322, 2009.

\bibitem{j:rcrdrt}
S.~JANSON.
\newblock Random cutting and records in deterministic and random trees.
\newblock {\em Random Struct. and Alg.}, 29:139--179, 2006.

\bibitem{k:spdtcpcgwt}
I.~KORTCHEMSKI.
\newblock A simple proof of Duquesne's theorem
on contour processes of conditioned Galton-Watson trees.
\newblock arXiv:1109.4138, 2011.

\bibitem{k:ilflpa}
A. KYPRIANOU.
\newblock {\em Introductory Lectures on Fluctuations of L\'evy
  Processes with Applications.}
\newblock Universitext, Springer, 2006.

\bibitem{mm:cdrt}
A.~MEIR and J.~MOON.
\newblock Cutting down random trees.
\newblock {\em J. Australian Math. Soc.}, 11:313--324, 1970.

\bibitem{t:rtsds}
W.F. TERHALLE.
\newblock R-trees and symmetric differences of sets.
\newblock {\em European J. Com binatorics}, 18:825--833, 1997.

\bibitem{v:dmfglt}
G.~VOISIN.
\newblock Dislocation measure of the fragmentation of a general {L}\'evy tree.
\newblock {\em ESAIM:PS}, 15:372--389, 2012.

\end{thebibliography}

\end{document}